\documentclass[a4paper,twoside]{amsart}
\usepackage[utf8]{inputenc}
\usepackage[hidelinks]{hyperref}
\usepackage{amsthm}
\usepackage{graphicx}

\usepackage{enumitem}

\theoremstyle{definition}
\newtheorem{definition}{Definition}[section]

\newtheorem{remark*}{Remark}
\newtheorem{remark}{Remark}[section]
\newtheorem*{claim*}{Claim}
\theoremstyle{remark}

\theoremstyle{plain}
\newtheorem{theorem}{Theorem}[section]
\newtheorem{corollary}[theorem]{Corollary}
\newtheorem{prop}[theorem]{Proposition}

\newtheorem{lemma}[theorem]{Lemma} 

\newtheorem{question}[theorem]{Question}

\def\dom{\mathop{\mathrm{Dom}}\nolimits}
\def\range{\mathop{\mathrm{Range}}\nolimits}
\def\Aut{\mathop{\mathrm{Aut}}\nolimits}

\def\Forb{\mathop{\mathrm{Forb}}\nolimits}

\def\str#1{\mathbf {#1}}

\def\K{{\mathcal K}}
\def\Fraisse{Fra\"{\i}ss\' e}

\begin{document}
\bibliographystyle{alpha}

\title{EPPA for two-graphs and antipodal metric spaces}

\authors{
\author[D. M. Evans]{David M. Evans}

\address{%
Department of Mathematics\\
Imperial College London\\
London SW7~2AZ\\
UK.}
\email{david.evans@imperial.ac.uk}
\author[J. Hubi\v cka]{Jan Hubi\v cka}
\address{Department of Applied Mathematics (KAM)\\ Charles University\\ Prague, Czech Republic}
\email{hubicka@kam.mff.cuni.cz}
\author[M. Kone\v cn\'y]{Mat\v ej Kone\v cn\'y}
\address{Charles University\\ Prague, Czech Republic}
\email{matej@kam.mff.cuni.cz}
\author[J. Ne\v set\v ril]{Jaroslav Ne\v set\v ril}

\address{Computer Science Institute of Charles University (IUUK)\\ Charles University\\ Prague, Czech Republic}
\email{nesetril@kam.mff.cuni.cz}
\thanks{The last three authors are supported by ERC Synergy grant DYNASNET 810115. Jan Hubi\v cka and Mat\v ej Kone\v cn\'y are supported by project 18-13685Y of the Czech Science Foundation (GA\v CR). Mat\v ej Kone\v cn\'y is also supported  by the Charles University Grant Agency (GA UK), project 378119.}

}

\begin{abstract}
We prove that the class of finite two-graphs has the extension property for partial
automorphisms (EPPA, or Hrushovski property), thereby answering a question of
Macpherson. In other words, we show that the class of graphs
has the extension property for switching automorphisms. We present a short, self-contained,
purely combinatorial proof which also proves EPPA for the class of integer valued antipodal metric spaces
of diameter 3, answering a question of Aranda et al.

The class of two-graphs is an important new example which behaves differently from all the
other known classes with EPPA: Two-graphs do not have the amalgamation
property with automorphisms (APA), their Ramsey expansion has to add a graph,
it is not known if they have coherent EPPA and even EPPA itself cannot be proved using the
Herwig--Lascar theorem.
\end{abstract}

\maketitle


\section{Introduction}
{\em Two-graphs}, introduced by G. Higman and studied extensively since the 1970s~\cite{Cameron1999,Seidel1973}, are 3-uniform
hypergraphs with the property that on every four vertices there is an even
number of hyperedges. 
A class $\mathcal C$ of finite structures (such as hypergraphs) has the \emph{extension property for partial automorphisms} (\emph{EPPA}, sometimes also called \emph{Hrushovski property}) if for every $\str A\in \mathcal C$ there exists $\str B\in \mathcal C$ containing $\str A$ as an (induced) substructure such that every isomorphism between substructures of $\str A$ extends to an automorphism of $\str B$. We call $\str B$ an \emph{EPPA-witness} for $\str A$.
We prove:
\begin{theorem}\label{thm:twographeppa}
The class $\mathcal T$ of all finite two-graphs has EPPA.
\end{theorem}
Our result answers a question of Macpherson which is also stated in Siniora's PhD
thesis~\cite{Siniora2} and can be seen as a contribution to the ongoing effort
of identifying new classes of structures with EPPA. This was started in 1992 by 
Hrushovski's proof~\cite{hrushovski1992} that the class of all finite
graphs has EPPA, and followed by a series of papers dealing with other classes, including~\cite{Aranda2017,Conant2015,Herwig1995,herwig1998,herwig2000,hodkinson2003,Hubicka2018metricEPPA,Hubicka2017sauer,Konecny2018b,otto2017,solecki2005,vershik2008}.

All proofs of EPPA in this paper are purely combinatorial and self-contained.
The second part of the paper requires some model-theoretical notions and discusses in more detail the interplay of the following properties for which there were no known examples before:
\begin{enumerate}
\item The usual procedure for building an EPPA-witness is to construct an incomplete
object (where some relations are missing) and later complete it to satisfy axioms
of the class without affecting any automorphisms (i.e. one needs to have
an \emph{automorphism-preserving completion}~\cite{Aranda2017}). This is not
possible for two-graphs and thus makes them related to tournaments which pose a
well known open problem in the area, see Remark~\ref{rem:noherwig}.

\item The class $\mathcal T$ does not have APA (\emph{amalgamation property with automorphisms}). Hodges, Hodkinson, Lascar, and Shelah~\cite{hodges1993b} introduced this notion and showed that APA together with EPPA imply the existence of \emph{ample generics} (see also~\cite[Chapter~2]{Siniora2}). To the authors' best knowledge, $\mathcal T$ is the only known class with EPPA but not APA besides pathological examples, see Section~\ref{sec:apa}.

\item In all cases known to the authors except for the class of all finite groups, whenever a class of structures $\mathcal C$ has EPPA then expanding a variant of $\mathcal C$ by linear orders gives a Ramsey expansion. This does not seem to be the case for two-graphs, see Section~\ref{sec:ramsey}.

\item Solecki and Siniora~\cite{Siniora,solecki2009} introduced the notion of \emph{coherent EPPA} (see Section~\ref{sec:coherence}) as a way to prove that the automorphism group of the respective \Fraisse{} limit contains a dense locally finite subgroup. Our method does not give coherent EPPA for $\mathcal T$ and thus it makes $\mathcal T$ the only known example with EPPA for which coherent EPPA is not known. However, our method does give coherent EPPA for the class of all antipodal metric spaces of diameter 3 and using it we are able to obtain a dense locally finite subgroup of the \Fraisse{} limit of $\mathcal T$, see Section~\ref{sec:groupmagic}.
\end{enumerate}

\medskip

Two-graphs are closely related to the switching classes of graphs and to double covers of complete graphs~\cite{Cameron1999,Seidel1973}, which is in fact key in this paper. Our results can thus be interpreted as a direct strengthening of the
theorem of Hrushovski~\cite{hrushovski1992} which states that the class of
all finite graphs $\mathcal G$ has EPPA. Namely we can consider $\mathcal G$
with a richer class of mappings --- the switching automorphisms.

Given a graph $\str{G}$ with vertex set $G$ and $S \subseteq G$, the
\emph{(Seidel) switching} $\str{G}_S$ of $\str{G}$ is the graph  created from $\str{G}$
by complementing the edges between $S$ and $G\setminus S$. (That is, for $s\in S$ and $t\in G\setminus S$ it holds that $\{s,t\}$ is an edge of $\str G_S$ if and only if $\{s,t\}$ is not an edge of $\str G$. Edges and non-edges with both endpoints in $S$ or $G\setminus S$ are preserved.)

Given a graph $\str{H}$ with vertex set $H$, a 
function $f\colon G\to H$ is a \emph{switching isomorphism} of $\str{G}$ and
$\str{H}$ if there exists $S\subseteq G$ such that $f$ is an isomorphism of 
$\str{G}_S$ and $\str{H}$. If $\str{G}=\str{H}$ we call such a function a
\emph{switching automorphism}.

\begin{definition}\label{def:eppsa}
We say that a class $\mathcal C\subseteq \mathcal G$ has the
\emph{extension property for switching automorphisms} if for every $\str{G}\in
\mathcal C$ there exists $\str {H}\in \mathcal C$ containing $\str G$ as an
induced subgraph such that every switching isomorphism of induced subgraphs of
$\str G$ extends to a switching automorphism of $\str H$
and, moreover, every isomorphism of induced subgraphs of $\str{G}$ extends to an automorphism of $\str H$.
\end{definition}

In this language we prove:
\begin{theorem}
\label{thm:switchings}
The class of all finite graphs $\mathcal G$ has the extension property for switching automorphisms.
\end{theorem}
Because of the `moreover' part of Definition~\ref{def:eppsa}, Theorem~\ref{thm:switchings}
implies the theorem of Hrushovski. It is also a strengthening of Theorem~\ref{thm:twographeppa}
by the following well-known correspondence between two-graphs and switching classes~\cite{Cameron1999,Seidel1973}.

\medskip

Given a graph $\str{G}$, its \emph{associated two-graph} $T(\str{G})$ is a two-graph on the
same vertex set as $\str{G}$ such that $\{a,b,c\}$ is a hyperedge if and only if the three-vertex subgraph induced by $\str G$ on $\{a,b,c\}$ has an odd number of edges.  Then a function $f\colon G \to H$ between graphs $\str{G}$ and $\str H$ is a switching isomorphism  if and only if it is an isomorphism between the associated two-graphs. Thus, the existence of a switching isomorphism
is an equivalence relation on the class of graphs and two-graphs correspond to the 
equivalence classes (called \emph{switching classes of graphs}).

\medskip

We shall see that the most natural setting for our proof is to work with the class of all finite
integer-valued antipodal metric spaces of diameter 3.
Following~\cite{Amato2016} we call a metric space an \emph{integer-valued metric space of diameter 3} if the distance of
every two distinct points is 1, 2 or 3. It is \emph{antipodal} if
\begin{enumerate}
 \item it contains no triangle with distances $2,2,3$, and
 \item the edges with label $3$ form a perfect matching (in other words, for
       every vertex there is precisely one \emph{antipodal vertex} at distance 3).
\end{enumerate}

We require that the domain and image of a partial automorphism of such a structure should be closed under taking antipodal points. However, this can be assumed without loss of generality because there always is a unique way of extending a partial automorphism not satisfying this condition to one which does.  In this language, we can then state our main theorem as:
\begin{theorem}\label{thm:antipeppa}
The class of all finite integer-valued antipodal metric spaces of diameter 3 has coherent EPPA.
\end{theorem}

This theorem also holds for all other antipodal metric spaces from Cherlin's
catalogue of metrically homogeneous graphs~\cite{Cherlin2013}, see~\cite{Konecny2019a}.
It answers affirmatively a question of Aranda, Bradley-Williams, Hubi\v cka,
Karamanlis, Kompatscher, Kone\v cn\'y and Pawliuk~\cite{Aranda2017} and completes
the analysis of EPPA for all classes from Cherlin's catalogue. However, in this note, for brevity, we often refer to an antipodal, integer-valued metric space of diameter 3 as an \emph{antipodal metric space}. Other antipodal metric spaces are not considered.

\section{Notation and preliminaries}
It is in the nature of this paper to consider multiple types of structures.
We will use bold letter such as $\str{A},\str{B},\str{C},\ldots$ to denote structures ((hyper)graphs or metric spaces defined below)
and corresponding normal letters, such as $A,B,C,\ldots$, to denote corresponding vertex sets. Our
substructures (sub-(hyper)graphs or subspaces) will always be induced.

Formally, we will consider a metric space to be a complete edge-labelled graph (that is, a complete graph where edges are labelled by the respective distances), or, equivalently, a relational structure with multiple binary relations representing the distances. This justifies that we will speak of pairs of vertices at distance $d$ as of \emph{edges of length $d$}. We will, however, use both notions (a vertex set with a distance function or a complete edge-labelled graph) interchangeably.
We adopt the standard notion of isomorphism, embedding and substructure.

\subsection{Coherent EPPA} Suppose $\mathcal C$ is a class of finite structures which has the hereditary and joint embedding properties. If $\mathcal C$ has EPPA, then it has the amalgamation property, so, assuming that there are only countably many isomorphism types of structures in $\mathcal C$, then we can consider the \Fraisse{} limit $\str M$ of $\mathcal C$. 
Coherence is a natural strengthening of EPPA which guarantees that $\Aut(\str{M})$ has a  dense locally finite subgroup~\cite{solecki2009,Siniora}. Note that the existence of a dense locally finite subgroup of the automorphism group of a homogeneous structure implies that its age has EPPA, but it is not known whether coherent EPPA also follows from this: see Section 5.1 of~\cite{Siniora}.
  At the moment all previously known EPPA classes are also
coherent EPPA classes. Interestingly, we can prove coherent EPPA for the antipodal metric spaces of diameter 3, but not for two-graphs (this is discussed in~Section~\ref{sec:groupmagic}). We need to introduce two additional
definitions.
\label{sec:coherence}
\begin{definition}[Coherent maps~\cite{solecki2009,Siniora}]
Let $X$ be a set and $\mathcal P$ be a family of partial bijections 
 between subsets
of $X$. A triple $(f, g, h)$ from $\mathcal P$ is called a {\em coherent triple} if $$\dom(f) = \dom(h), \range(f ) = \dom(g), \range(g) = \range(h)$$ and $$h = g \circ f.$$

Let $X$ and $Y$ be sets, and $\mathcal P$ and $\mathcal Q$ be families of partial bijections between subsets
of $X$ and between subsets of $Y$, respectively. A function $\varphi\colon \mathcal P \to \mathcal Q$ is said to be a
{\em coherent map} if for each coherent triple $(f, g, h)$ from $\mathcal P$, its image $(\varphi(f), \varphi(g), \varphi(h))$ in $\mathcal Q$ is coherent. In the case of EPPA, where the elements of $\mathcal Q$ are automorphisms of $Y$, we sometimes refer to the image of $\varphi$ as a coherent family of automorphisms extending $\mathcal P$.
\end{definition}
\begin{definition}[Coherent EPPA~\cite{solecki2009,Siniora}]
\label{defn:coherent}
A class $\mathcal C$ of finite structures is said to have {\em coherent EPPA} if $\mathcal C$ has EPPA and moreover the extension of partial automorphisms
is coherent. More precisely, for every $\str{A} \in \mathcal C$, there exists an EPPA-witness $\str{B} \in \mathcal C$ for $\str A$ and a coherent map $f \mapsto \hat{f}$ from the family of partial automorphisms of $\str A$ to the group of automorphisms of $\str B$. In this case we also call $\str{B}$ a \emph{coherent EPPA-witness} of $\str{A}$.
\end{definition}

\section{EPPA for antipodal metric spaces}
Given an antipodal metric space $\str{A}$ we give a direct construction of a
coherent EPPA-witness $\str{B}$. Some ideas are based on a construction
of Hodkinson and Otto~\cite{hodkinson2003} and some of the terminology is
loosely based on Hodkinson's exposition of this construction~\cite{hodkinson}.
Note that our techniques also give a very simple and short proof of EPPA for graphs.

Fix a (finite) antipodal metric space $\str{A}$.  
Denote by $M=\{e_1,e_2,\ldots, e_n\}$ the set of all edges of $\str{A}$ of length 3.
For a function $\chi\colon M\to\{0,1\}$ we denote by $1-\chi$ the function satisfying $(1-\chi)(e)=1-\chi(e)$ for every $e\in M$.
We construct $\str{B}$ as follows:
\begin{enumerate}
  \item The vertices of $\str{B}$ are all pairs $(e,\chi)$ where $e\in M$ and $\chi$ is a function
from $M$ to $\{0,1\}$ (called a \emph{valuation function}).
  \item Distances for $(e,\chi)\neq (f,\chi')\in B$ are given by the following rules:
    \begin{enumerate}[label=(\roman*)]
       \item\label{C:a} $d_\str{B}((e,\chi),(e,1-\chi))=3$,
       \item\label{C:b} $d_\str{B}((e,\chi),(f,\chi'))=1$ if and only if $\chi(f)=\chi'(e)$,
       \item\label{C:c} $d_\str{B}((e,\chi),(f,\chi'))=2$ otherwise.
    \end{enumerate}
\end{enumerate}

\begin{lemma}
The structure $\str{B}$ is an antipodal metric space.
\end{lemma}
\begin{proof}
Given $(e,\chi)\in B$, by \ref{C:a} we know that there is precisely one vertex at distance 3 (namely $(e,1-\chi)$)
 and thus the edges of length 3 form a perfect matching.

It remains to check that  every quadruple
$(e,\chi),(e,1-\chi),(f,\chi'),(f,1-\chi')$ of distinct vertices of $\str{B}$
is an antipodal metric space.
By \ref{C:b} we know that precisely one of $(f,\chi'),(f,1-\chi')$ is at distance 1 from $(e,\chi)$
and by \ref{C:c} that the other is at distance 2, similarly for $(e, 1-\chi)$. It also follows that $d_{\str B}((e,\chi),(f,\chi'))=d_{\str B}((e,1-\chi),(f,1-\chi'))$ and $d_{\str B}((e,\chi),(f,1-\chi'))=d_{\str B}((e,1-\chi),(f,\chi'))$.
\end{proof}
We now define an embedding $\psi\colon \str{A}\to \str{B}$ and refer to its image
 $\str{A}'$  in $\str{B}$ as a \textit{generic copy} of $\str{A}$ in $\str{B}$.

Fix an arbitrary function $p\colon A\to \{0,1\}$ such that whenever $d_\str{A}(x,x')=3$, then $p(x)=1-p(x')$. This function partitions the vertices of $\str{A}$ into two \emph{podes} such that
pairs of vertices at distance 3 are in different podes. For every $1\leq i\leq n$ we denote by $x_i$ and $y_i$ the endpoints of $e_i$ such that $p(x_i)=0$ and $p(y_i)=1$.
We construct $\psi$ by putting $\psi(x_i) = (e_i, \chi_i)$ and $\psi(y_i) = (e_i, 1-\chi_i)$, where $\chi_i$ is defined as
$$
\chi_i(e_j) = 
\begin{cases}
  0 & \text{if }j\geq i \\
  0 & \text{if }j < i\text{ and } d_{\str A}(x_i, x_j) = 1 \\
  1 & \text{otherwise.}
\end{cases}
$$
It follows from the construction that $\psi$ is indeed an embedding $\str A\to \str B$. We put $\str A'=\psi(\str A)$. Now we are ready to show the main result of this section:

\begin{prop}
\label{prop:tmain}
With the above notation, the antipodal metric space $\str{B}$ is a coherent EPPA-witness for $\str{A}'$.
Moreover, $p\circ \psi^{-1}$ extends to a function $\hat p\colon B\to \{0,1\}$ such that whenever a partial automorphism $\varphi$ of $\str{A}'$ preserves values of $p\circ\psi^{-1}$, then its coherent extension $\theta$ preserves values of $\hat p$.
\end{prop}
\begin{proof}
Let $\varphi$ be a partial automorphism of $\str{A}'$. Let $\pi\colon B \to M$ be
the {projection} mapping $(e,\chi) \mapsto e$.  By this projection, $\varphi$
induces a partial permutation of $M$ and we denote by $\hat\varphi$ an extension of it to a
permutation of $M$. To obtain coherence we always extend the permutation in an
\emph{order-preserving way}, that is, we enumerate $M\setminus\dom(\varphi) = \{e_{i_1},\ldots,e_{i_k}\}$ and $M\setminus\range(\varphi) = \{e_{j_1},\ldots,e_{j_k}\}$, where $i_1 < \cdots < i_k$ and $j_1 < \cdots < j_k$, and put $\hat\varphi(e_{i_\ell}) = e_{j_\ell}$ for every $1\leq \ell\leq k$ (cf. ~\cite{Siniora,solecki2009}).

Let $F$ be the set consisting of unordered pairs $\{e,f\}$, $e,f\in M$ (possibly $e=f$), such that there exists $\chi$
with the property that $(e,\chi)\in \dom(\varphi)$ and $\chi (f)\neq \chi' (\hat\varphi(f))$
for $(\hat\varphi(e),\chi')=\varphi ((e,\chi))$. We say that these pairs are \emph{flipped} by $\varphi$. Because of the choice of $\str A'$, there are zero or  two choices for $\chi$ for every $e$, and if there are two, then they are $\chi$ and $1-\chi$ for some $\chi$ and both of them give the same result. Note that there may be no $\eta$ such that $(f,\eta)\in \dom(\varphi)$.


Define a function $\theta\colon B\to B$ by putting
$$\theta((e,\chi))=(\hat\varphi(e),\xi)$$ where 
$$
\xi(\hat\varphi(f))=
\begin{cases}
  \chi(f) &  \text{if }\{e,f\}\notin F \\
  1-\chi(f) & \text{if }\{e,f\}\in F.
\end{cases}
$$

\medskip

First we verify that $\theta$ extends $\varphi$.  Suppose $(e,\chi) \in \dom(\varphi)$. Write  $\theta(e,\chi) = (\hat\varphi(e),\xi)$ and $\varphi(e,\chi) = (\hat\varphi(e),\chi')$. We must check that $\xi = \chi'$, so we let $f \in M$ and show that $\chi'(\hat\varphi(f)) = \xi(\hat\varphi(f))$. But this follows easily from the definitions of $\xi$ and $F$, by considering the cases $\{e,f\} \in F$ and $\{e,f\} \not\in F$ separately.

\medskip

Now we check that $\theta$ is an automorphism of $\str{B}$. It is
easy to see that $\theta$ is one-to-one (one can construct its inverse) and that it preserves antipodal pairs. To check that $d_\str B((e,\chi),(f,\eta)) = d_\str B(\theta((e,\chi)),\theta((f,\eta)))$ for non-antipodal pairs, denote $\theta((e,\chi)) = (\hat\varphi(e),\chi')$ and $\theta((f,\eta)) = (\hat\varphi(f),\eta')$.

By definition of $\theta$, we have $\chi(f)\neq \chi'(\hat\varphi(f))$ if and only if $\{e,f\}\in F$, and analogously, $\eta(e)\neq \eta'(\hat\varphi(e))$ if and only if $\{e,f\}\in F$. Putting this together (and again considering the cases $\{e,f\} \in F$ and $\{e,f\} \not\in F$ separately) we get $\chi(f) =  \eta(e)$ if and only if $\chi'(\hat\varphi(f))= \eta'(\hat\varphi(e))$. Together with the definition of $d_{\str B}$ this implies that indeed $d_\str B((e,\chi),(f,\eta)) = d_\str B(\theta((e,\chi)),\theta((f,\eta)))$.

\medskip

Thus far, we have shown that $\str{B}$ is an EPPA-witness for $\str{A}'$.

\medskip

Now we put $\hat p((e,\chi))=\chi(e)$. Note that $\hat p(\psi(x_i)) = \chi_i(e_i) = 0=p(x_i)$ and similarly $\hat p(\psi(y_i)) = p(y_i)$. So $\hat p$ extends $p \circ\psi^{-1}$.

Suppose $\varphi$ preserves values of $p$. Then for $(e,\chi)\in \dom(\varphi)$ it holds that $\chi(e) = \chi'(\hat\varphi(e))$, where $(\hat\varphi(e),\chi')=\varphi((e,\chi))$. Thus there is no $e\in M$ such that $\{e\}\in F$ (for such an $e$, there would have to be some $(e,\chi) \in \dom\varphi$, of course). By definition of $\theta$ we immediately get $\hat p((e,\chi))=\hat p(\theta ((e,\chi)))$ for every $(e,\chi)\in B$.

\medskip

Finally we verify coherence for the above construction.
Consider partial automorphisms $\varphi_1$, $\varphi_2$ and $\varphi$ of 
$\str{A}'$ such that $\varphi$ is the composition of $\varphi_1$ and $\varphi_2$.
Denote by $\hat\varphi_1$, $\hat\varphi_2$ and $\hat\varphi$ their
corresponding permutations of $M$ constructed above. Let   $F_1$, $F_2$ and $F$ be the corresponding sets of flipped pairs
and  $\theta_1$, $\theta_2$ and $\theta$ the corresponding extensions. Because the 
permutations $\hat\varphi_1$, $\hat\varphi_2$ and $\hat\varphi$ were constructed by extending projections of $\varphi_1$, $\varphi_2$ and $\varphi$  (which also compose coherently) in an order-preserving way, we know that
$\hat\varphi$ is the composition of $\hat\varphi_1$ and $\hat\varphi_2$.  

To see
that $\theta$ is the composition of $\theta_1$ and $\theta_2$ one first checks that for $e \in \pi(\dom(\varphi)) = \pi(\dom(\varphi_1))$ and $f \in M$ one has that $\{e,f\} \in F$ if and only if
\[(\{e,f\} \not\in F_1 \mbox{ and } \{\hat\varphi_1(e),\hat\varphi_1(f)\} \not\in F_2) \mbox{ or } (\{e,f\} \in F_1 \mbox{ and } \{\hat\varphi_1(e),\hat\varphi_1(f)\} \in F_2).\]

The definition of $\theta_1, \theta_2$ and $\theta$ then gives the required result. In more detail, write $\theta_1(e,\chi) = (\hat\varphi_1(e), \xi_1)$ and $\theta_2\theta_1((e,\chi)) = \theta_2(\hat\varphi_1(e), \xi_1) = (\hat\varphi(e),\xi_2)$, where, for $f \in M$,
$$
\xi_2(\hat\varphi_2(\hat\varphi_1(f)))=
\begin{cases}
  \xi_1(\hat\varphi_1(f)) &  \text{if }\{\hat\varphi_1(e),\hat\varphi_1(f)\}\not\in F_2 \\
  1-\xi_1(\hat\varphi_1(f)) &  \text{if }\{\hat\varphi_1(e),\hat\varphi_1(f)\}\in F_2.
\end{cases}
$$
Applying the definition of $\xi_1$ and using the above observation finishes the calculation.

%
%
%
%

\medskip

\end{proof}

\section{Proofs of the main results}
\label{sec:equivalence}
Theorem~\ref{thm:antipeppa} is a direct consequence of Proposition~\ref{prop:tmain}. Next we use the correspondence between graphs with switching automorphisms and antipodal metric spaces (or, in the language of Seidel, double-covers of complete graphs~\cite{Cameron1999,Seidel1973}) to prove Theorem~\ref{thm:switchings}:

\begin{proof}[Proof of Theorem~\ref{thm:switchings}]
Given a graph $\str{G}$, we construct an antipodal metric space $\str{A}$ 
on vertex set $G\times\{0,1\}$ with distances defined as follows:
\begin{enumerate}
\item $d_\str{A}((x,0),(x,1))=3$ for every $x\in G$,
\item $d_\str{A}((x,i),(y,i))=1$ for every $x\neq y$ forming
an edge of $\str{G}$ and $i\in \{0,1\}$,
\item $d_\str{A}((x,i),(y,1-i))=1$ for every $x\neq y$ forming a non-edge of $\str{G}$ and $i\in \{0,1\}$, and
\item $d_\str{A}((x,i),(y,j))=2$ otherwise.
\end{enumerate}

Let $p\colon A\to \{0,1\}$ be a function defined by $p((x,i))=i$. We apply
Proposition~\ref{prop:tmain} to construct an antipodal metric space $\str{C}$ 
and a function $\hat p\colon C\to \{0,1\}$. Construct a graph $\str{H}$ with vertex set $\{x\in C:\hat{p}(x)=0\}$ with
$x,y$ forming an edge if and only if $d_\str{C}(x,y)=1$. 

Now consider a partial automorphism $\varphi$ of $\str{G}$. This automorphism
corresponds to a partial automorphism $\varphi'$ of $\str{A}$ by putting
$\varphi'((x,i))=(\varphi(x),i)$ for every $x\in \dom(\varphi)$ and $i\in
\{0,1\}$. $\varphi'$ then extends to $\theta$ which preserves values of
$\hat p$. Consequently, $\theta$ restricted to $H$ is an automorphism of
$\str{H}$.

Finally consider a partial switching automorphism $\varphi$ (i.e. a switching isomorphism of induced subgraphs). Let $S$ be the set
of vertices switched by $\varphi$.  Now the partial automorphism of $\str{A}$
is defined by putting $\varphi'((x,i))=(\varphi(x),i)$ if $x\notin S$ and
$\varphi'((x,i))=(\varphi(x),1-i)$ otherwise. Again extend $\varphi'$ to $\theta$
and observe that $\theta$ gives a switching automorphism of $\str{H}$.
\end{proof}

EPPA for two-graphs follows easily too:
\begin{proof}[Proof of Theorem~\ref{thm:twographeppa}]
Let $\str{T}$ be a finite two-graph, pick an arbitrary vertex $x\in T$ and define a graph $\str G$ on the vertex set $T$ such that $\{y,z\}\in E_\str G$ if and only if $x\notin \{y,z\}$ and $\{x,y,z\}$ is  a triple of $\str T$. Observe that $\str T = T(\str G)$. By Theorem~\ref{thm:switchings}, there is a graph $\str H$ containing $\str G$ such that every switching isomorphism of induced subgraphs of $\str G$ extends to a switching automorphism of $\str H$. We claim that $T(\str H)$ is an EPPA-witness for $\str T$.

To prove that, let $\varphi\colon \str T\to \str T$ be a partial automorphism of $\str T$. By the construction of $\str G$ and the correspondence between graphs and two-graphs, $\varphi$ is a switching isomorphism of subgraphs of $\str G$ induced on the domain and range of $\varphi$ respectively. Hence, it extends to a switching automorphism of $\str H$ and thus an automorphism of $T(\str H)$, which is what we wanted.
\end{proof}

\begin{remark}
Observe that the EPPA-witness given in this proof of Theorem~\ref{thm:twographeppa}
is not necessarily a coherent EPPA-witness. The problem lies in the proof of Theorem~\ref{thm:switchings}, because for every switching isomorphism of subgraphs of $\str G$ there are two corresponding partial automorphisms of $\str A$. For example, if the partial automorphism of $\str G$ is a partial identity, one partial automorphism of $\str A$ is also a partial identity, while the other flips all the involved edges of length 3.
While the first is extended to the identity by the construction in Proposition~\ref{prop:tmain}, the
other is extended to a non-trivial permutation of the edges of length 3. This issue carries over to Theorem~\ref{thm:twographeppa}, and in fact, it seems
to be a fundamental obstacle for using antipodal metric spaces to prove coherent EPPA for two-graphs.
\end{remark}

\section{Existence of a dense locally finite subgroup}
\label{sec:groupmagic} 
As discussed in Section~\ref{sec:coherence} here, Solecki and Siniora~\cite{Siniora,solecki2009} introduced the notion of \emph{coherent EPPA} for a \Fraisse{} class as a way to prove that the automorphism group of the respective \Fraisse{} limit contains a dense locally finite subgroup. While we cannot prove coherent EPPA for $\mathcal T$, Theorem~\ref{thm:antipeppa} gives coherent EPPA for the class of all antipodal metric spaces of diameter 3 and thus there is a dense locally finite subgroup of the automorphism group of its \Fraisse{} limit. We now show how this gives a dense, locally finite subgroup of the automorphism group of the \Fraisse{} limit of $\mathcal T$.

Let $\mathbb T$ be the \Fraisse{} limit of $\mathcal T$ and $\mathbb M$ be the \Fraisse{} limit of the class of all finite antipodal metric spaces of diameter $3$. The main result of this section is the following theorem.
\begin{theorem}\label{thm:groups}
There is a dense, locally finite subgroup of $\Aut(\mathbb T)$.
\end{theorem}
In order to prove this, we first describe how to get a two-graph from an antipodal metric space, which is again a well-known construction~\cite{Cameron1999,Seidel1973}.

\begin{figure}                                                                  
\centering                                                                      
\includegraphics{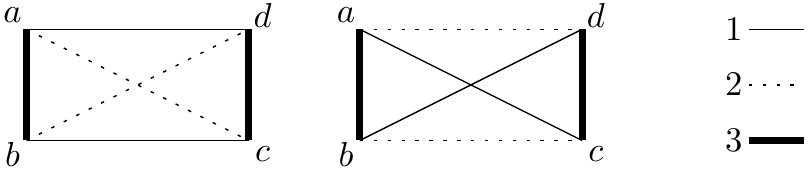}                                                
\caption{Two possible antipodal quadruples for choice of $a,b,c,d$, $d_\str{A}(a,b)=d_\str{A}(c,d)=3$.}
\label{fig:antipodalquadruple}                                                  
\end{figure}
In an antipodal metric space $\str{A}$, every quadruple of distinct vertices $a,b,c,d$ such that $d_\str{A}(a,b)=d_\str{A}(c,d)=3$
(a pair of edges of label 3) induces one of the two (isomorphic) subspaces depicted in Figure~\ref{fig:antipodalquadruple} --- we call these \emph{antipodal quadruples}. However, three edges with label $3$ can induce two non-isomorphic structures; the edges of length 1 either form two triangles, or one 6-cycle (see Figure~\ref{fig:doublecovers}). This motivates the following correspondence:
\begin{figure}                                                                  
\centering                                                                      
\includegraphics{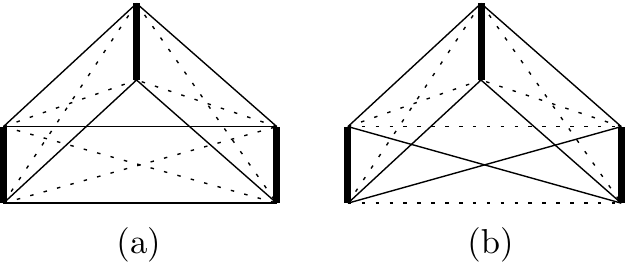}                                                
\caption{Two non-isomorphic antipodal metric spaces with 6 vertices.}
\label{fig:doublecovers}                                                  
\end{figure}

\begin{definition}[From antipodal spaces to two-graphs]
Let $\str A$ be an antipodal metric space. Let $M$ be the set of all edges of $\str A$ of length 3 (thus, $|M|=\frac{|A|}{2}$ if $\str A$ is finite). Define $T(\str A)$ to be the 3-uniform hypergraph on vertex set $M$ where $\{a,b,c\}$ is a hyperedge if and only if the substructure of $\str A$ induced on the edges $a,b,c$ is isomorphic to that depicted in Figure~\ref{fig:doublecovers}~(b). 
\end{definition}
It is straightforward to verify that $T(\str A)$ is a two-graph. Clearly there is a natural two-to-one map $\str{A} \to T(\str{A})$ and this induces a group homomorphism $\Aut(\str{A}) \to \Aut(T(\str{A}))$. 
It is also well-known that there is a converse to this construction (usually expressed in terms of double covers of compete graphs). Suppose $\str T$ is a two-graph. Let $\str G$ be a graph in the switching class of $\str T$ (so $T(\str{G}) = \str{T}$) and let $\str A$ be the antipodal metric space constructed in the proof of Theorem~\ref{thm:switchings}. Then $T(\str{A}) = \str{T}$ and the proof of Theorem~\ref{thm:switchings} shows that the map $\Aut(\str{A}) \to \Aut(\str{T})$ is surjective, as any automorphism of $\str{T}$ is a switching automorphism of $\str{G}$. More generally, a similar argument gives the following well-known fact:

\begin{lemma}\label{lem:lifting}
 If $\str{A}_1, \str{A}_2$ are antipodal metric spaces and $\beta\colon T(\str{A}_1) \to T(\str{A}_2)$ is an isomorphism of two-graphs, then there is an isomorphism $\alpha\colon \str{A}_1 \to \str{A}_2$ which induces $\beta$.
\end{lemma}

We can now give:


\begin{proof}[Proof of Theorem~\ref{thm:groups}] Let $\str{T}$ be the two-graph $T(\mathbb{M})$. By Lemma~\ref{lem:lifting}, any isomorphism between finite substructures of $\str{T}$ lifts to an isomorphism between finite substructures of $\mathbb{M}$ and so, as $\mathbb{M}$ is homogeneous, this partial isomorphism is induced by an automorphism of $\mathbb{M}$. This shows that $\str{T}$ is homogeneous. The construction of an antipodal metric space from a two-graph shows that $\str{T}$ embeds every finite two-graph, and therefore $\str{T}$ is isomorphic to $\mathbb{T}$. So we have, again using Lemma~\ref{lem:lifting}, a surjective homomorphism $\alpha\colon \Aut(\mathbb{M}) \to \Aut(\mathbb{T})$. As already observed, there is a dense, locally finite subgroup $H$ of $\Aut(\mathbb{M})$. Then $\alpha(H)$ is a dense, locally finite subgroup of $\Aut(\mathbb{T})$.

%
\end{proof}

\section{Amalgamation property with automorphisms}
\label{sec:apa}
Let $\str{A}$, $\str{B}_1$ and $\str{B}_2$ be structures, $\alpha_1$ an
embedding of $\str{A}$ into $\str{B}_1$ and $\alpha_2$ an embedding of
$\str{A}$ into $\str{B}_2$. Then every structure $\str{C}$ with embeddings
$\beta_1\colon\str{B}_1 \to \str{C}$ and $\beta_2\colon\str{B}_2\to\str{C}$ such that
$\beta_1\circ\alpha_1 = \beta_2\circ\alpha_2$ is called an \emph{amalgamation}
of $\str{B}_1$ and $\str{B}_2$ over $\str{A}$ with respect to $\alpha_1$ and
$\alpha_2$.  Amalgamation is \emph{strong} if
$\beta_1(x_1)=\beta_2(x_2)$ if and only if $x_1\in \alpha_1(A)$ and $x_2\in
\alpha_2(A)$.

For simplicity, in the following definition we will assume that all the embeddings
in the definition of amalgamation are in fact inclusions.
\begin{definition}[APA]\label{defn:apa}
Let $\mathcal C$ be a class of finite structures. We say that $\mathcal C$ has the \emph{amalgamation property with automorphisms} (\emph{APA}) if for every $\str A, \str B_1, \str B_2 \in \mathcal C$ such that $\str A\subseteq \str B_1,\str B_2$ there exists $\str C\in \mathcal C$ which is an amalgamation of $\str B_1$ and $\str B_2$ over $\str A$, has $\str B_1,\str B_2\subseteq \str C$ and furthermore the following holds:

For every pair of automorphisms $f_1\in\Aut(\str B_1)$, $f_2\in \Aut(\str B_2)$ such that $f_i(A) = A$ for $i\in\{1,2\}$ and $f_1|_A = f_2|_A$, there is an automorphism $g\in \Aut(\str C)$ such that $g|_{B_i} = f_i$ for $i\in\{1,2\}$.
\end{definition}

In other words, APA is a strengthening of the amalgamation property which requires that every pair of automorphisms of $\str B_1$ and $\str B_2$ which agree on $\str A$ extends to an automorphism of $\str C$.

As was mentioned in the introduction, EPPA + APA imply the existence of ample generics~\cite{hodges1993b} and it turns out that most of the known EPPA classes also have APA. We now show that this is not the case for two-graphs.

\begin{figure}                                                                  
\centering                                                                      
\includegraphics{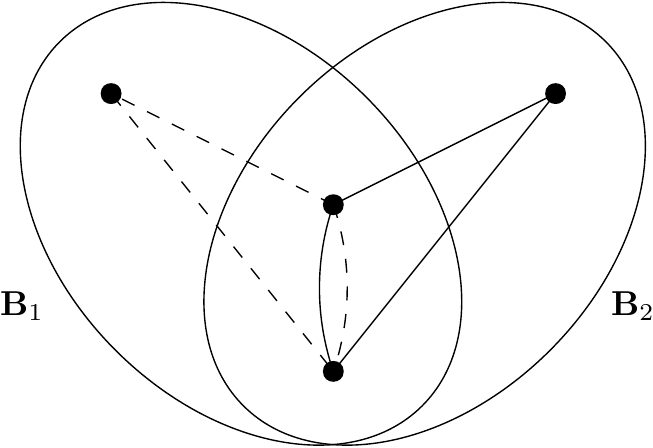}                                                
\caption{A failure of APA for two-graphs}
\label{fig:noapa}
\end{figure}
\begin{prop}
Let $\str A$ be the two-graph on two vertices, $\str B_1$ be the two-graph on three vertices with no hyper-edge and $\str B_2$ be the two-graph on three vertices which form a hyper-edge. It is possible to amalgamate $\str B_1$ and $\str B_2$ over $\str A$, but it is not possible to amalgamate them with automorphisms.
\end{prop}
\begin{proof}
For convenience, we name the vertices as in Figure~\ref{fig:noapa}: $A=\{u,v\}$, $B_1 = \{u,v,x_1\}$ and $B_2= \{u,v,x_2\}$. For $i\in \{1,2\}$ let $f_i$ be the automorphism of $\str B_i$ such that $f_i(x_i) = x_i$, $f_i(u)=v$ and $f_i(v)=u$. Clearly $f_1$ and $f_2$ agree on $\str A$.

Consider the amalgamation problem for $\str B_1$ and $\str B_2$ over $\str A$ (with inclusions and with automorphisms $f_1, f_2$) and assume for a contradiction that there is $\str C\in\mathcal T$ and an automorphism $g$ of $\str C$ as in Definition~\ref{defn:apa}. By the definition of $\mathcal T$, we get that there has to be an even number of triples in $\str C$ on $\{u,v,x_1,x_2\}$ and since we know that $\{u,v,x_1\}$ is not a triple and $\{u,v,x_2\}$ is a triple, there actually have to be precisely two triples on $\{u,v,x_1,x_2\}$. Therefore, exactly one of $\{u,x_1,x_2\}$ and $\{v,x_1,x_2\}$ has to form a triple in $\str C$. But this means that $g$ is not an automorphism (as $g$ fixes $x_1$ and $x_2$ and swaps $u$ and $v$), a contradiction.

On the other hand, if we only want to amalgamate $\str B_1$ and $\str B_2$ over $\str A$ (not with automorphisms), then this is clearly possible.
\end{proof}

\begin{remark}\label{rem:whynoapa}
In the introduction we mentioned that there are also some pathological examples with EPPA but not APA:
\begin{enumerate}
\item Let $\str M$ be the two-vertex set with no structure. Its age consists of the empty set, the one-element set and $\str M$. Consider the amalgamation problem for $\str A$ the empty set and $\str B_1=\str B_2=\str M$. The amalgam has to be $\str M$ again. But then it is impossible to preserve all four possible pairs of automorphisms of $\str B_1$ and $\str B_2$.

This phenomenon clearly happens because this is not a strong amalgamation class (and is exhibited by other non-strong amalgamation classes) and, indeed, disappears when we only consider closed structures.

\item Let now $\str M$ be the disjoint union of two infinite cliques, that is, an equivalence relation with two equivalence classes and consider its age. Let $\str A$ be the empty structure and $\str B_1 = \str B_2$ consist of two non-equivalent vertices. This amalgamation problem, again, does not have an APA-solution, because one needs to decide which pairs of vertices will be equivalent and this cannot preserve all four pairs of automorphisms.

This generalises to situations where the homogeneous structure has a definable equivalence relation with finitely many equivalence classes (or, more generally, of finite index).

However, the reason here is that the equivalence classes are algebraic in a quotient, i.e. a similar reason as in the previous point. One can either require the amalgamation problem to specify which classes go to which ones or, equivalently, consider an expansion which \emph{weakly eliminates imaginaries} and the problem disappears.
\end{enumerate}
While these two examples point out that, at least from the combinatorial point of view, one needs a more robust definition for APA, two-graphs seem to \textit{innately not have APA}.
\end{remark}

\section{Ramsey expansion of two-graphs}
\label{sec:ramsey}
As was pointed out recently, the methods for proving EPPA and the Ramsey property share many similarities, see e.g.~\cite{Aranda2017}. The standard strategy is to study the \emph{completion problem} for given classes and their expansions, see~\cite{Aranda2017,Evans3,Konecny2018b,Hubicka2017sauer}. EPPA is usually a corollary of one of the steps towards finding a Ramsey expansion.

The situation is very different for two-graphs. As we shall observe in this section, the Ramsey question can be easily solved using standard techniques, while for EPPA this is not the case (see Remark~\ref{rem:noherwig}). This makes two-graphs an important example of the limits of the current methods and shows that the novel approach of this paper is in fact necessary.

We now give the very basic
definitions of the structural Ramsey theory. 

A class $\K$ of structures is {\em Ramsey} if for every $\str{A},\str{B}\in \K$ there
exists $\str{C}\in \K$ such that for every colouring of copies of $\str{A}$ in $\str{C}$
there exists a copy of $\str{B}$ in $\str{C}$ that is monochromatic.
(By a {\em copy} of $\str{A}$ in $\str{C}$ we mean any substructure of $\str{C}$ isomorphic to $\str{A}$.) This statement is abbreviated as $\str C\longrightarrow (\str B)_k^\str A$.

If class $\mathcal K$ has {\em joint embedding property} (which means that for
every $\str{A}_1, \str{A}_2\in \mathcal K$ there is $\str{C}\in \mathcal K$ which contains a copy of both $\str{A}_1$ and $\str{A}_2$)
then EPPA of $\mathcal K$ implies amalgamation.  It is well known that Ramsey
property also implies amalgamation
property~\cite{Nevsetvril1989a,Nevsetril2005} under the assumption of joint embedding.

Every Ramsey class consists of rigid structures, i.e. structures with no
non-trivial automorphism. The usual way to establish rigidity is to extend the
language (in a model-theoretical way) by an additional binary relation $\leq$ which fixes the
ordering of vertices.  It is thus a natural question whether the class
$\overrightarrow{\mathcal T}$ of all two-graphs with a linear ordering of
the vertices is Ramsey.  We now show that the answer is negative and to do it, we need the following statement.
\begin{prop}
\label{prop:twographexp}
For every two-graph $\str{A}$ there exists a two-graph $\str{B}$
such that every graph in the switching class of $\str{B}$ contains a
copy of every graph in the switching class of $\str{A}$.
\end{prop}
\begin{proof}
Denote by $\str{G}$ the disjoint union of all graphs in the switching class of $\str{A}$.
Now let $\str{H}$ be a graph such that every colouring of vertices
of $\str{H}$ by 2 colours induces a monochromatic copy of $\str{G}$ (that is,
$\str H$ is vertex-Ramsey for $\str G$) --- it exists by a theorem of Folkman~\cite{folkman1970,Nevsetvril1976b,Nevsetvril1977}).

Every graph $\str{H}'$ in the
switching class of $\str{H}$ induces a colouring of
vertices of $\str{H}$ by two colours: $\str H'$ being in the switching class of
$\str H$ means that there is a set $S\subseteq H$ such that $\str H' = \str H_S$, the
colour classes are then $S$ and $H\setminus S$ respectively.

By the construction of $\str H$ we find a copy $\widetilde{\str{G}}\subseteq \str H$ of $\str G$
which is monochromatic with respect to this colouring. This however implies that 
the graphs induced by $\str{H}$ and $\str{H}'$ on the set $\widetilde{G}$ are isomorphic
and thus $\str H'$ indeed contains every graph in the switching class of $\str A$. 
Therefore, we can put $\str{B}$ to be the two-graph associated to $\str{H}$. 
\end{proof}
\begin{corollary}
The class $\overrightarrow{\mathcal T}$ is not Ramsey for colouring pairs of vertices.
\end{corollary}
\begin{proof}
Let $\str A$ be the two-graph associated to an arbitrary graph containing both
an edge and a non-edge, and let $\str B$ be the two-graph given by
Proposition~\ref{prop:twographexp} for $\str A$. Let $\overrightarrow{\str B}$ be an arbitrary
linear ordering of $\str B$.

Assume that there exists an ordered two-graph $\overrightarrow{\str{C}}$ such that
$$\overrightarrow{\str{C}}\longrightarrow(\overrightarrow{\str{B}})^{\overrightarrow{\str{E}}}_2$$
where $\overrightarrow{\str{E}}$ is the unique ordered two-graph on 2 vertices. 

Let $\overrightarrow{\str I}$ be an arbitrary graph from the switching class of $\overrightarrow{\str C}$
(with the inherited order) and colour copies of $\overrightarrow{\str E}$ red if they correspond to an edge
of $\overrightarrow{\str I}$ and blue otherwise. By the construction of $\str B$ it follows
that there is no monochromatic copy of $\overrightarrow{\str B}$, a contradiction.
\end{proof}
Proposition~\ref{prop:twographexp} says that the expansion of two-graphs adding a particular graph from the switching
class has the so-called \emph{expansion property}. As a
consequence of the Kechris--Pestov--Todor\v cevi\' c correspondence~\cite{Kechris2005}, one gets
that every Ramsey expansion of $\mathcal T$ has to fix a particular representative of the
switching class (see e.g.~\cite{NVT14} for details, it is also easy to see this directly). On the other hand, expanding any two-graph by a linear
order and a particular graph from the given switching class is a Ramsey expansion by the
Ne\v set\v ril-R\"odl theorem~\cite{Nevsetvril1977b}.

\section{Remarks}
\begin{remark}\label{rem:noherwig}
The by-now-standard strategy for proving EPPA for class $\mathcal C$ in, say, a relational language $L$ may be summarized as follows:
\begin{enumerate}
\item Assume that, every pair of vertices of every structure in $\mathcal C$ is in some relation. If it is not, we can always add a new binary relation to $L$ and put all pairs into the relation.
\item\label{noherwig:2} Study the class $\mathcal C^-$ which consists of all $L$-structures $\str A^-$ such that there is $\str A\in \mathcal C$ which is a \emph{completion} of $\str A^-$, that is, $\str A^-$ and $\str A$ have the same vertex set $A$ and there is $X\subset \mathcal P(A)$, a subset of the powerset of $A$, such that if $Y\in X$ and $Z\subseteq Y$, then $Z\in X$, and for each relation $R\in L$ it holds that $R^{\str A^-} = R^{\str A}\cap X$.
\item\label{noherwig:3} Find a finite family of $L$-structures $\mathcal F$ such that $\mathcal C^-$ is precisely $\Forb(\mathcal F)$, that is, the class of all finite $L$-structures $\str B$ such that there is no $\str F\in \mathcal F$ with a homomorphism to $\str B$.
\item Prove that in fact for every $\str A^-\in \mathcal C^-$ there is $\str A\in \mathcal C$ which is its \emph{auto\-morphism-preserving completion}, that is, $\str A^-$ can be obtained from $\str A$ as in point~\ref{noherwig:2} and furthermore $\str A^-$ and $\str A$ have the same automorphisms.
\item Use the Herwig--Lascar theorem~\cite{herwig2000} omitting homomorphisms from $\mathcal F$ to get EPPA-witnesses in $\Forb(\mathcal F)$.
\item Take the automorphism-preserving completion of the witnesses to get EPPA-witnesses in $\mathcal C$ and thus prove EPPA for $\mathcal C$.
\end{enumerate}
In various forms, this strategy was applied, for example, in~\cite{solecki2005,Conant2015,Aranda2017,Konecny2018b,Hubicka2017sauer}.
See also~\cite{Hubicka2016} where the notion of completions was introduced.

As we have seen in Section~\ref{sec:apa}, $\mathcal T$ does not admit automorphism-preserving completions (because APA is a weaker property). One can also prove (and it will appear elsewhere), using the negative result of Proposition~\ref{prop:twographexp} and Theorem~2.11 of~\cite{Hubicka2016}, that $\mathcal T$ cannot be described by finitely many forbidden homomorphisms (hence in particular there is no finite family $\mathcal F$ satisfying point \ref{noherwig:3} above). This we believe is the first time the Ramsey techniques have been used to prove a negative result for EPPA.
\end{remark}

\begin{remark}
$\mathcal T$ is one of the five reducts of the random graph~\cite{Thomas1991}. Besides $\mathcal T$, the random graph itself and the countable set with no structure, the remaining two corresponding automorphism groups can be obtained by adding an isomorphism between the random graph and its complement and an isomorphism between the generic two-graph and its complement respectively.

By a similar argument, one can prove that the ``best'' Ramsey expansion (that is, with the expansion property) of these structures is still the ordered random graph. On the other hand, EPPA for these two classes is an open problem (we conjecture that neither of these two classes has EPPA).
\end{remark}

\begin{remark}
Theorem~\ref{thm:switchings} implies the following. For every graph $\str G$ there exists an EPPA-witness $\str H$ with the property that the two-graph associated to $\str H$ is an EPPA-witness for the two-graph associated for $\str G$, in other words, it implies that the class of all graphs and two-graphs respectively are a non-trivial positive example for the following question.
\end{remark}
\begin{question}
For which pairs of classes $\mathcal C$, $\mathcal C^-$ such that $\mathcal C^-$ is a reduct of $\mathcal C$ does it hold that for every $\str A\in \mathcal C$ there is $\str B\in \mathcal C$ such that $\str B$ is an EPPA-witness for $\str A$ (in $\mathcal C$) and furthermore if $\str A^-$ and $\str B^-$ are the corresponding reducts in $\mathcal C^-$ then $\str B^-$ is an EPPA-witness for $\str A^-$ (in $\mathcal C^-$)?
\end{question}

\begin{remark}
As was already mentioned, ample generics are usually proved to exist by showing the combination of EPPA and APA. Siniora in his thesis~\cite{Siniora2} asked if two-graphs have ample generics. This question still remains open, although we conjecture that it is not the case (ample generics are equivalent to having the so-called \emph{weak amalgamation property for partial automorphisms} and the \emph{joint embedding property for partial automorphisms} and it seems that the reasons for two-graphs not having APA are fundamental enough to also hold in the weak amalgamation context).\footnote{Added August 2019: in joint work with P. Simon, the authors have shown that this conjecture holds.}
\end{remark}

\section*{Acknowledgment}
We would like to thank the anonymous referee for their helpful comments.


\bibliography{ramsey.bib}

\end{document}